\documentclass{amsart}

\usepackage{amscd}
\usepackage{amssymb}
\usepackage{graphicx}
\usepackage{tikz}
\usepackage{color}

\usepackage{amsthm}

\newtheorem{thm}[equation]{Theorem}

\newtheorem{prop}[equation]{Proposition}

\newtheorem*{thm*}{Theorem}
\newtheorem*{prop*}{Proposition}
\newtheorem*{cor*}{Corollary}
\newtheorem*{lem*}{Lemma}
\newtheorem*{MT*}{Main Theorem}
\newtheorem{ques}[equation]{Question}
\newtheorem*{ques*}{Question}

\theoremstyle{definition} %

\newtheorem*{defn*}{Definition}

\newtheorem{eg}[equation]{Example}

\theoremstyle{remark} %

\newtheorem*{rmk*}{Remark}
\newtheorem*{rmks*}{Remarks}

\newtheoremstyle{exercise}
  {3pt}
  {3pt}
  {\small}
  {}
  {\sc\small}
  {.}
  {.5em}
   {}     
  {}

\theoremstyle{exercise}

\makeatletter
{\renewcommand{\theequation}{#1}}%
{\renewcommand{\theequation}{\arabic{equation}}\addtocounter{equation}{-1}\global\@ignoretrue}
\makeatother

\makeatletter
{\renewcommand{\theequation}{#1}\begin{eqnarray}}%
{\end{eqnarray}\renewcommand{\theequation}{\arabic{equation}}\addtocounter{equation}{-1}\global\@ignoretrue}
\makeatother

\makeatletter
{\smallskip \refstepcounter{equation}\noindent{\textbf{\theequation.} }{{\textbf{#1.}}}}%
{\smallskip \global\@ignoretrue}
\makeatother

\makeatletter
{\smallskip \refstepcounter{equation}\noindent{\textbf{\theequation.} }{{\textbf{#1}}}}%
{\smallskip \global\@ignoretrue}
\makeatother

\makeatletter
{\smallskip \refstepcounter{equation}{\sc \theequation}{\sc (#1).}}%
{\smallskip \global\@ignoretrue}
\makeatother

\makeatletter
{\smallskip \refstepcounter{equation}\noindent{\sc \theequation.}{\sl{ #1.}}}%
{\smallskip \global\@ignoretrue}
\makeatother

\makeatletter
\newenvironment{borel*}%
{\smallskip \refstepcounter{equation}\noindent{\textbf{\theequation.}}}%
{\global\@ignoretrue}
\makeatother

\newcommand{\flist}[1]{\hangindent\leftmargini\textup{(1)}\hskip\labelsep {#1}%
\begin{enumerate}%
\setcounter{enumi}{1}%
}

\newcommand{\ot}{\otimes}

\DeclareMathOperator{\Lie}{Lie}

\newcommand{\C}{{\mathbb{C}}}        
\newcommand{\Q}{{\mathbb{Q}}}        
\newcommand{\R}{{\mathbb{R}}}        
\newcommand{\F}{{\mathbb{F}}}        
\newcommand{\Z}{{\mathbb{Z}}}        

\newcommand{\e}{\varepsilon}

\newcommand{\la}{\lambda}
\newcommand{\D}{\Delta}

\newcommand{\G}{{\Gamma}}       

\newcommand{\oddots}{{\mathinner{\mkern1mu\raise1pt\vbox{\kern7pt\hbox{.}}\mkern2mu\raise4pt\hbox{.}\mkern2mu\raise7pt\hbox{.}\mkern1mu}}}

\newcommand{\qform}[1]{{\left\langle{#1}\right\rangle}}                   

\DeclareMathOperator{\Spin}{Spin}           
\DeclareMathOperator{\Sp}{Sp}

\DeclareMathOperator{\SL}{SL}
\DeclareMathOperator{\GL}{GL}
\DeclareMathOperator{\PGL}{PGL}

\newcommand{\SO}{\mathrm{SO}}

\newcommand{\so}{\mathfrak{so}}

\newcommand{\Gm}{\mathbb{G}_m}

\DeclareMathOperator{\Tr}{Tr}

\DeclareMathOperator{\ed}{ed}

\DeclareMathOperator{\car}{char}

\DeclareMathOperator{\Aut}{Aut}

\newcommand{\Hom}{{\mathrm{Hom}}}

\newcommand{\injects}{\hookrightarrow}

\renewcommand{\e}{\mathfrak{e}}
\newcommand{\f}{\mathfrak{f}}
\newcommand{\gl}{\mathfrak{gl}}

\DeclareMathOperator{\Inv}{Inv}

\DeclareMathOperator{\PSL}{PSL}

\renewcommand{\sl}{\mathfrak{sl}}

\newcommand{\Sets}{\mathsf{Sets}}

\usepackage[all]{xy}

\usepackage[hyphenbreaks]{breakurl}

\numberwithin{equation}{section}

\newcommand{\g}{\mathfrak{g}}
\DeclareMathOperator{\rad}{rad}

\newcommand{\Gad}{G^{\mathrm{ad}}}

\newcommand{\E}{\mathsf{E}}

\renewcommand{\sc}{\mathrm{sc}}

\begin{document}

\title{$E_8$, the most exceptional group}
\author{Skip Garibaldi}
\address{Center for Communications Research, San Diego, California 92121}
\email{skip \text{at} member.ams.org}
\begin{abstract}
The five exceptional simple Lie algebras over the complex number are included one within the other as $\g_2 \subset \f_4 \subset \e_6 \subset \e_7 \subset \e_8$.  The biggest one, $\e_8$, is in many ways the most mysterious.  This article surveys what is known about it including many recent results, focusing on the point of view of Lie algebras and algebraic groups over fields.
\end{abstract}


\maketitle

\setcounter{tocdepth}{1}
\tableofcontents

\section{Introduction}

The Lie algebra $\e_8$ or Lie group $E_8$ was first sighted by a human being sometime in summer or early fall of 1887, by Wilhelm Killing as part of his program to classify the semisimple finite-dimensional Lie algebras over the complex numbers \cite[pp.~162--163]{Hawkins}.  Since then, it has been a source of fascination for mathematicians and others in its role as the largest of the exceptional Lie algebras. (It appears, for example, as part of the fictional Beard-Einstein Conflation in the prize-winning novel \emph{Solar} \cite{Solar}.)  This article will discuss some of the many new mathematical discoveries that have been obtained over the last few years and point out some of the things that remain unanswered.

Killing's classification is now considered the core of a typical graduate course on Lie algebras, and the paper containing the key ideas, \cite{Killing2}, has even been called ``the greatest mathematical paper of all time''\footnote{Any such delcaration is obviously made with an intent to provoke.  One objection I have heard to this particular nomination is that Killing made a serious mathematical error --- made famous by Cartan's quote ``Malheureusement les recherches de M.~Killing manquent de rigueur, et notamment, en ce qui concerne les 
groupes qui ne sont pas simples, il fait constamment usage d'un th\'eor\`eme qu'il ne d\'emontre pas dans sa g\'en\'eralit\'e'' \cite[p.~9]{Ca:th} --- but this error is from a different paper and is irrelevant to the current discussion.  See \cite{Hawkins} for details.}, see \cite{Coleman:greatest} and \cite{Helgason:ex}.
Killing showed that the simple Lie algebras make up four infinite families together with just five\footnote{Since Killing's paper contains so many new and important ideas, produced by someone working in near-isolation, it feels churlish to mention that it actually claims that there are \emph{six} exceptional ones.   Killing failed to notice that two of the ones in his list were isomorphic.} others, called  \emph{exceptional}.  They are ordered by inclusion and $\e_8$ is the largest of them.  When someone says ``$E_8$'' today, they might be referring to the simple Lie algebra $\e_8$, an algebraic group or Lie group, a specific rank 8 lattice, or a collection of 240 points in $\R^8$ (the  root system $\E_8$). 
This article gives an introduction to these objects as well as a survey of some of the many results on $E_8$ that have been discovered since the millennium and indications of the current frontiers.

There are many reasons to be interested in $E_8$.  Here is a pragmatic one.  The famous Hasse-Minkowski Principle for quadratic forms and the Albert-Brauer-Hasse-Noether Theorem for division algebras over number fields can be viewed as special cases of the more general Hasse Principle for semisimple algebraic groups described and proved in \cite[Ch.~6]{PlatRap}.  This is a powerful theorem, which subsumes not just those famous results but also local-global statements for objects that don't have ready descriptions in elementary language.  Its proof involves case-by-case considerations, with the case of $E_8$ being the hardest one, eventually proved by Chernousov in \cite{Ch:hasse}.  
That $E_8$ was hardest is typical.  In the words of \cite{Leeuwen:Vogan}: ``Exceptional groups, and in particular $E_8$, appear to have a more dense and complicated structure than classical ones, making computational problems more challenging for them.... Thus $E_8$ serves as a `gold standard': to judge the effectiveness of the implementation of a general algorithm, one looks to how it performs for $E_8$.''

I have heard this sentiment expressed more strongly as: We understand semisimple algebraic groups only as far as we understand $E_8$.  If we know some statement for all groups except $E_8$, then we do not really know it.

Here is a more whimsical reason to be interested in $E_8$: readers who in their youth found quaternions and octonions interesting will naturally gravitate to the exceptional Lie algebras and among them the largest one, $\e_8$.  Indeed, one might call $E_8$ \emph{the Monster of Lie theory}, because it is the largest of the exceptional groups, just as the Monster is the largest of the sporadic finite simple groups.  Another take on this is that ``$E_8$ is the most non-commutative of all simple Lie groups'' \cite{Lusztig:E8}.

\subsection*{Remarks on exceptional Lie algebras}
The exceptional\footnote{We are following the usual definition of ``exceptional'', but there are alternatives that are appealing.  For example, \cite{DeligneGross} considers a longer chain of inclusions
\[
\sl_2 \subset \sl_3 \subset \g_2 \subset \so_8 \subset \f_4 \subset \e_6 \subset \e_7 \subset \e_8
\]
such that, at the level of simply connected Lie groups, all of the inclusions are unique up to conjugacy.  One could alternatively define all the Lie algebras appearing in that chain to be exceptional.}
 Lie algebras mentioned above form a chain 
\[
\g_2 \subset \f_4 \subset \e_6 \subset \e_7 \subset \e_8.
\]
There are various surveys and references focusing on some or all of these, such as \cite{Baez}, \cite{FF}, \cite{Adams:ex}, \cite{Jac:ex}, \cite{KMRT}, \cite{Frd:OAO}, and \cite{Sp:ex}.  Certainly, the greatest amount of material is available for $\g_2$, corresponding to the octonions, where one can find discussions as accessible as \cite{Numbers}.  Moving up the chain, less is available and what exists is somewhat less accessible.  This note takes the approach of jumping all the way to $\e_8$ at the end, because it is the case with the least existing exposition, or, to say the same thing differently, the most opportunity.

Another feature of exceptional groups is that, by necessity, various ad hoc constructions are often employed in order to study one group or another. These can appear inexplicable at first blush.  As much as we can, we will explain how these peculiar construction arise naturally from the general theory of semisimple Lie algebras and groups.  See \S\ref{gradings} for an illustration of this.

\section{What is $E_8$?} \label{whatis}

To explain what Killing was doing, suppose you want to classify Lie groups, meaning for the purpose of this section, a smooth complex manifold $G$ that is also a group and that the two structures are compatible in the sense that multiplication $G \times G \to G$ and inversion $G \to G$ are smooth maps.  The tangent space to $G$ at the identity, $\g$, is a vector space on which $G$ acts by conjugation, and this action gives a bilinear map $\g \times \g \to \g$ denoted $(x,y) \mapsto [xy]$ such that $[xx] = 0$ and $[x[yz]] + [y[zx]] + [z[xy]] = 0$ for all $x, y, z \in \g$.  The vector space $\g$ endowed with this ``bracket'' operation is called the \emph{Lie algebra} of $G$.  Certainly, if two Lie groups are isomorphic, then so are their Lie algebras; the converse, or something like it, also holds and is known as ``Lie's third theorem''.

Consequently, for understanding Lie groups, it is natural to classify the finite-dimensional Lie algebras over the complex numbers, which was Killing's goal.  Suppose $L$ is such a thing.  Then, among the ideals $I$ such that $[II] \subsetneq I$, there is a unique maximal one, called the radical of $L$ and denoted $\rad L$.   The possible such ideals $\rad L$ are unclassifiable if $\dim(\rad L)$ is large enough, see \cite{BLS}, so we ignore it; replacing $L$ by $L/\rad L$ we may assume that $\rad L = 0$.  Then $L$ is a direct sum of simple Lie algebras, which we now describe how to analyze.  

\subsection*{The root system} To classify the simple Lie algebras, Killing observed that we can extract from a simple $L$ a piece of finite combinatorial data called a \emph{root system}, and he classified the possible root systems.  The simple root system called $\E_8$ is a 240-element subset $R$ of $\R^8$, whose elements are called \emph{roots}.  It consists of the short vectors in the lattice $Q$ generated by the \emph{simple roots} 
\begin{gather*}
\alpha_1 = \frac12 (e_1 - e_2 - e_3 - e_4 - e_5 - e_6 - e_7 + e_8), \quad \alpha_2 = e_1 + e_2, \quad \alpha_3 = -e_1 + e_2, \\
 \alpha_4 = -e_2 + e_3, \quad \alpha_5  = -e_3 + e_4, \quad \alpha_6 = -e_4 + e_5, \\
 \alpha_7 = -e_5 + e_6, \quad \alpha_8 = -e_6 + e_7
\end{gather*}
where $e_i$ denotes the $i$-th element of an orthonormal basis of $\R^8$.  
This information is encoded in the graph
\setlength{\unitlength}{0.05cm}
\newsavebox{\Eviii}
\savebox{\Eviii}(90,21){\begin{picture}(90,21)
\put(0,6){\circle*{3}}
\put(0,6){\line(1,0){15}}
\put(15,6){\circle*{3}}
\put(15,6){\line(1,0){15}}
\put(30,6){\circle*{3}}
\put(30,6){\line(1,0){15}}
\put(45,6){\circle*{3}}
\put(30,6){\line(0,1){15}} 
\put(60,6){\circle*{3}}
\put(30,21){\circle*{3}}
\put(75,6){\circle*{3}}
\put(75,6){\line(1,0){15}}
\put(90,6){\circle*{3}}
\put(45,6){\line(1,0){15}}
\put(60,6){\line(1,0){15}}
\end{picture}}
\begin{equation} \label{dynkin}
\begin{picture}(90,21)(0,0)
 \put(0,0){\usebox{\Eviii}} 
\put(-2,0){\mbox{\small $\alpha_1$}}
\put(13,0){\mbox{\small $\alpha_3$}}
\put(28,0){\mbox{\small $\alpha_4$}}
\put(43,0){\mbox{\small $\alpha_5$}}
\put(58.5,0){\mbox{\small $\alpha_6$}}
\put(73,0){\mbox{\small $\alpha_7$}}
\put(88,0){\mbox{\small $\alpha_8$}}
\put(32.5,20){\mbox{\small $\alpha_2$}}
\end{picture}
\end{equation}
where vertices correspond to simple roots, a single bond joining $\alpha_i$ to $\alpha_j$ indicates that $\alpha_i \cdot \alpha_j = -1$, and no bond indicates that $\alpha_i \cdot \alpha_j = 0$.   Such a graph is called a \emph{Dynkin diagram}, and the root system is commonly described in this language today, but they were not invented for more than 50 years after Killing's paper.  The graph determines the root system up to isomorphism.  Evidently the particular embedding of $R$ in $\R^8$ and the labeling of the simple roots is not at all unique; here we have followed \cite{Bou:g4}.  One could equally well draw it as suggested in \cite{Lusztig:E8}, with the edges lying on along the edges of a cube as in Figure \ref{dynkin.cube}.

\begin{figure}
\begin{tikzpicture}
\draw [thick/.style={line width=2pt}, 
       line cap=round]                
 (0,0) edge [thick]  (2,0)
 (2,0) edge [dashed] (2,2)
 (2,2) edge [thick]  (0,2)
 (0,2) edge [thick]  (0,0)

 (1,1) edge [thick]  (3,1)
 (3,1) edge [dashed] (3,3)
 (3,3) edge [dashed] (1,3)
 (1,3) edge [dashed] (1,1)

 (0,0) edge [dashed] (1,1)
 (2,0) edge [thick]  (3,1)
 (2,2) edge [thick]  (3,3)
 (0,2) edge [thick]  (1,3) 
; 
\end{tikzpicture}
\caption{Dynkin diagram of $E_8$ obtained by deleting 5 edges from a cube in $\R^3$} \label{dynkin.cube}
\end{figure}
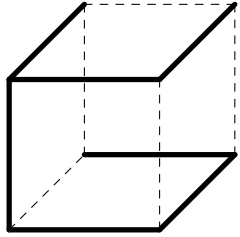

\begin{figure}[bht]
\includegraphics[width=0.42\textwidth]{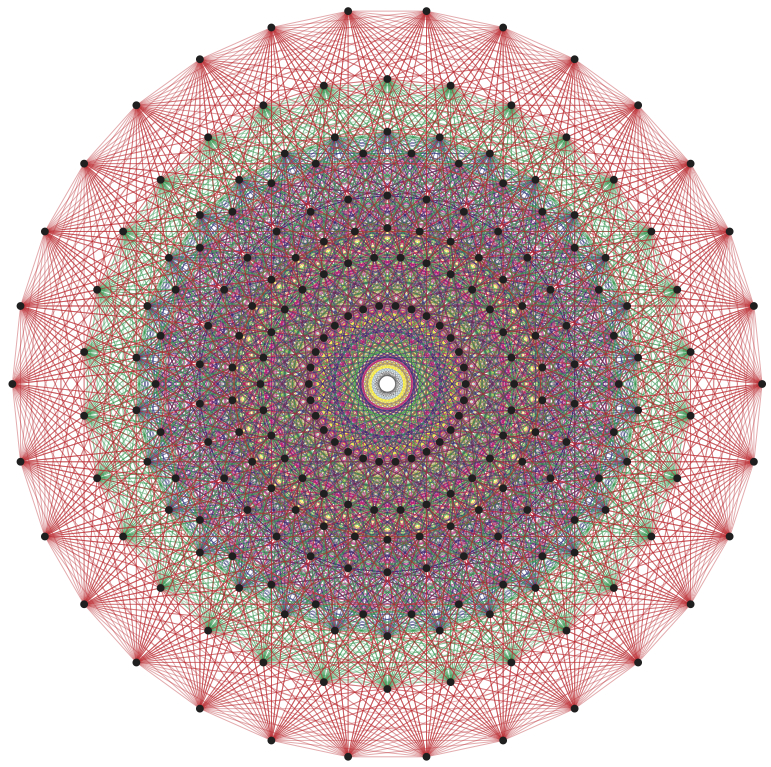}
\caption{$E_8$'s publicity photo, courtesy of J.~Stembridge} \label{publicity}
\end{figure}

We can draw a picture of the full set $R$ by projecting $\R^8$ onto a well-chosen plane\footnote{Precisely: take a Coxeter element $w$ in the Weyl group of $R$.  It is unique up to conjugacy and its minimal polynomial has a simple factor $x^2 - 2\cos(2\pi/h) + 1$, for $h = 30$ the Coxeter number for the root system $\E_8$.  That is, there is a unique plane in $\R^8$ on which the projection of $w$ has that minimal polynomial, so that $w$ acts on the plane by rotations by $2\pi/h$.}, see Figure \ref{publicity}.  In the picture, the 240 roots are the black dots, which lie, 30 each, on 8 concentric circles --- this much follows from general theory as in \cite[VI.1.11, Prob.~33(iv)]{Bou:g4}.  The edges in the picture join the images of roots that are nearest neighbors.  The colors don't have any mathematical significance.

\subsection*{The $E_8$ lattice} The lattice $Q$ together with the quadratic form $v \mapsto \frac12 \|v\|^2$ is the unique positive-definite, even, unimodular lattice of rank 8.  It has been known for a long time that it is the lattice with the densest sphere packing in $\R^8$, and the sphere packing with the highest kissing number \cite{ConwaySloane}.  Recently, Viazovska showed that it gives the densest sphere packing in $\R^8$, among all possible packings both lattice and non-lattice \cite{Viazovska:E8}.  At the time of her work, the densest packing was unknown in $\R^n$ for all $n \ge 4$!

The roots $R$ --- the short vectors in $Q$ --- are a 7-design on the unit sphere in $\R^8$, i.e., every polynomial on $\R^8$ of total degree at most 7 has the same average on $R$ as on the entire sphere, with the minimal number of elements \cite{BannaiSloane}; such a thing is very rare.  The theta-series for this lattice, $\theta(q) = \sum_{v \in Q} q^{\|v\|^2/2}$, is the 4-th Eisenstein series, which provides a connection with the $j$-invariant via the formula $j(q) = \theta(q)^3/\eta(q)^{24}$, where $\eta$ denotes Dedekind's eta-function.\footnote{For connections of the coefficients of $j(q)^{1/3}$ with dimensions of irreducible representations of $E_{8,\C}$, see \cite[esp.~\S4.1]{HeMcKay}.}
For more on this lattice and how it fits into the rest of mathematics, see \cite{ConwaySloane} or  \cite{Elkies1}.

Another view on the lattice is that it is a maximal order in the (real) octonions, as described in \cite{ConwaySmith}.  In particular, this turns $Q$ into a non-associative ring with 240 units, namely the roots $R$.

\subsection*{From the root system to the Lie algebra}
Starting from $R \subset \R^8$, one can write down a basis and multiplication table for the Lie algebra $\e_8$ over $\C$ as in \cite{Lusztig:can} or \cite{Geck}.  (This is a recent development.  The old way to do this is via the Chevalley relations as in, for example, \cite[\S{VIII.4.3}]{Bou:g7}.  Some signs have to be chosen --- see \cite{Ti:struct}, \cite{Carter:simple}, or \cite[\S2.3]{FrenkelKac} --- but the isomorphism class of the resulting algebra does not depend on the choices made.  Explicit multiplication tables are written in \cite[p.~328]{Cartan:real} and \cite{Vavilov:const}.)  In this way, one can make the Lie algebra from the root system.

\section{$E_8$ as an automorphism group}

To study $\e_8$, then, we can view it as a specific example of a simple Lie algebra described by generators and relations given by the root system.  We will exploit this view below.  However, that is not what is commonly done for other simple Lie algebras and groups!  It is typical to describe the ``classical'' (meaning not exceptional)  Lie algebras and groups not in terms of their root systems, but rather as, for example, $\SL_n$, the $n$-by-$n$ determinant 1 matrices; $\SO_n$, the $n$-by-$n$ orthogonal matrices of determinant 1; and $\Sp_{2n}$, the $2n$-by-$2n$ symplectic matrices.  Killing explicitly asked, in his 1889 paper, for similar descriptions of the exceptional Lie algebras and groups.

Looking back at the descriptions of the classical groups in the previous paragraph, each of them arises from a faithful irreducible representation $G \injects \GL(V)$ where $\dim V$ is smaller than $\dim G$.  Using now the classification of irreducible representations of simple Lie algebras, the Weyl dimension formula, and other general-purpose tools described entirely in terms of root systems (and which post-date Killing and can be found in standard textbooks such as \cite{Hum:Lie}), we find that the smallest faithful irreducible representations of $\g_2$, $\f_4$, $\e_6$, and $\e_7$ (alternatively, the simply connected Lie groups $G_2$, $F_4$, $E_6$, and $E_7$) have dimensions 7, 26, 27, and 56, which are much smaller than the dimensions 14, 52, 78, and 133 of the corresponding algebra.  
In the years since Killing posed his problem, all of these algebras have found descriptions using the corresponding representation, as we now describe.

The smallest nontrivial representation of $G_2$, call it $V$, has dimension 7.  Calculating with weights, one finds that $V$ has $G_2$-invariant linear maps $b \!: V \ot V \to \C$ (a bilinear form, which is symmetric) and $\times \!: V \ot V \to V$ (product, which is skew-symmetric\footnote{The product can be viewed as a 7-dimensional analogue of the usual cross-product in $\R^3$, compare for example \cite[\S10.3]{Numbers}.}), which are unique up to multiplication by an element of $\C^\times$.  To see this, one decomposes the $G_2$-modules $V^* \ot V^*$ and $V^* \ot V^* \ot V$ into direct sums of irreducible representations and notes that in each case there is a unique 1-dimensional summand.   The subgroup of $\GL(V)$ preserving these two structures is $G_2$.  Alternatively, we can define a $G_2$ invariant bilinear form $t$ and product on $\C \oplus V$ by setting
\begin{gather*}
t((x,c), (x',c')) := xx' + b(c,c') \ \text{and}\\
(x,c) \bullet (x',c') := (xx' + b(c,c'), xc' + x'c + c \times c').
\end{gather*}
With these definitions and scaling $\times$ so that $b(u \times v, u\times v) = b(u,u) b(v,v) - b(u,v)^2$,  $\C \oplus V$ is isomorphic to the complex octonions, a nonassociative algebra --- see \cite[\S10.3]{Numbers} for details --- and $G_2$ is the automorphism group of that algebra, compare \cite[\S2.3]{Sp:ex}.  Equivalently, $\g_2$ is the Lie algebra of $\C$-linear derivations of $\C \oplus V$.

In a similar way, the smallest nontrivial representation of $F_4$, call it $V$, has unique $F_4$-invariant linear maps $b \!: V \ot V \to \C$ (a symmetric bilinear form) and  $\times \!: V \ot V \to V$ (product, which is symmetric).  Extending in a similar manner, we find that $\C \oplus V$ has the structure of an Albert algebra --- a type of Jordan algebra --- and $F_4$ is the automorphism group of that algebra, see 
 \cite{ChevSchaf} or \cite[\S7.2]{Sp:ex}.  
 
The smallest nontrivial representation of $E_6$, call it $V$, has an $E_6$-invariant cubic polynomial $f \!: V \to \C$ that is unique up to multiplication by an element of $\C^\times$.  The isometry group of $f$, i.e., the subgroup of $g \in \GL(V)$ such that $f \circ g = f$, is $E_6$.  This was pointed out in \cite{Ca:th}, but see \cite[\S7.3]{Sp:ex} for a proof.

For $E_7$, the smallest nontrivial representation has an $E_7$-invariant quartic polynomial whose isometry group is generated by $E_7$ and the group $\qform{e^{\pi i/2}}$ of fourth roots of unity \cite{Frd:E7}.  That is, it has two connected components, and the connected component of the identity is $E_7$.

\subsection*{Answers for $E_8$}
The same approach, applied to $E_8$, is problematic.  The smallest nontrivial representation of $E_8$ has dimension 248; it is the action on its Lie algebra $\e_8$.  In this context $E_8$ is the automorphism group of the Lie algebra\footnote{Or, what is the same, the group of linear transformations preserving both the Killing form $\kappa \!: \e_8 \times \e_8 \to \C$ and the alternating trilinear form $\wedge^3 \e_8 \to \C$ given by the formula $x \wedge y \wedge z \mapsto \kappa(x,[yz])$.} $\e_8$.  This is not specific to $E_8$ --- it is a typical property of semisimple Lie algebras, see \cite{St:aut} for a precise statement ---and only serves to re-cast problems about the algebras as problems about the group and vice versa.

However, there is a general result that, in the case of $E_8$, says that for each faithful irreducible representation $V$, there is a homogeneous, $E_8$-invariant polynomial $f$ such that $E_8$ is the identity component of the stabilizer of $f$ in $\GL(V)$.

In the case of the smallest faithful irreducible representation, there is a degree 8 homogeneous polynomial on the Lie algebra $\e_8$ whose automorphism group is generated by $E_8$ and the eighth roots of unity $\qform{ e^{\pi i/4}}$.  (For experts: one takes a degree 8 generator for the Weyl-group-invariant polynomials on a Cartan subalgebra and pulls it back to obtain an $E_8$-invariant polynomial on all of $\e_8$.  The paper \cite{CederwallPalmkvist} gives a formula for it in terms of invariants of the $D_8$ subgroup of $E_8$.)  

The second smallest faithful irreducible representation $V$ has dimension 3875.  Combinatorial calculations with ``weights'' of the representation show that it has a nonzero, $E_8$-invariant symmetric bilinear form $b$ and trilinear form $t$, both of which are unique up to scaling by an element of $\C^\times$.  These give a commutative and $E_8$-invariant ``product'', i.e., a bilinear map, $\bullet \!: V \times V \to V$ defined by $t(x,y,z) = b(x, y \bullet z)$.  (Close readers of Chevalley will recognize this construction of a product from \cite[\S4.2]{Chev}.)  The automorphism group of $(V, \bullet)$ is $E_8$, see \cite{GG:simple}.

These results are recent and have not yet been seriously exploited.  

\subsection*{As an automorphism of varieties} 
As with other semisimple groups, one can view $E_8$ as the group of automorphisms of a projective  variety on which it acts transitively --- a \emph{flag variety} --- see \cite{Dem:aut} for the general statement.  In the case of $E_8$, the smallest such variety has dimension 57 (and can be obtained by quotienting out by the maximal parabolic subgroup determined by the highest root \cite[p.~152]{Ca:th}); see \cite{MR1846901} for an alternative treatment.

\section{Constructing the Lie algebra via gradings} \label{gradings}

Earlier, we described constructing the Lie algebra $\e_8$ by generators and relations using a Chevalley basis.  There are various refinements on this that we now discuss.  This situation is quite general and is often applied to study other exceptional groups, so we write in a slightly more general context.  Consider an irreducible root system $R$ with set $\D$ of simple roots, and let $\g$ be the Lie algebra generated from this using a Chevalley basis.

Suppose that $\g$ has a grading by an abelian group $\G$, meaning that as a vector space $\g$ is a direct sum of subspaces $\g_\gamma$ for $\gamma \in \G$ and that $[g_\gamma, g_{\gamma'}] \in \g_{\gamma + \gamma'}$ for $\gamma, \gamma' \in \G$.  Then $\g_0$ is a Lie subalgebra of $\g$ and each $\g_\gamma$ is a $\g_0$-module.  Roughly speaking, one can typically write down a formula for the bracket on $\g$ in terms of the action of $\g_0$ on the $\g_\gamma$, see \cite[Ch.~6]{Adams:ex} or \cite[\S22.4]{FH} for discussion.  The literature contains a profusion of such constructions, which are sometimes called a ``$\g_0$ construction of $\g$''.  (I will give some references below, but it is hopeless to attempt to be comprehensive.)

To produce such a construction of $\g$, then, we should look for gradings of $\g$.  These are controlled by subgroups of $\Aut(\g)$ in the following sense.  If $F = \C$ and $\g$ has a $\Z/n$ grading, then the map $\sum_\gamma y_\gamma \mapsto \sum e^{2\pi i \gamma/n} y_\gamma$ exhibits the $n$-th roots of unity $\mu_n(\C) = \qform{e^{2\pi i/n}}$ as a subgroup of $\Aut(\g)$ and conversely every homomorphism $\mu_n \to \Aut(\g)$ gives a $\Z/n$ grading on $\g$ by the same formula.  Generalizing this example somewhat to the language of diagonalizable group schemes as in \cite{SGA3.2} or \cite{Wa} and writing $\G^D$ for the Cartier dual of $\G$, we have the following folklore classification of gradings:
\begin{prop}
Let $R$ be a commutative ring and $\g$ a finitely generated Lie algebra over $R$.  Then for each finitely generated abelian group $\G$, there is a natural bijection between the set of 
$\G$-gradings on $\g$ and the set of morphisms $\G^D \to \Aut(\g)$ of group schemes over $R$.
\end{prop}

\begin{proof}[Sketch of proof]
The group $\G$ is naturally identified with the dual $\Hom(\G^D, \Gm)$, so given a $\G$-grading on on $\g$, setting $t \cdot y = \gamma(t)y$ for $y \in \g_\gamma$ and $t \in \G^D$ defines a homomorphism $\G^D \to \Aut(\g)$. Conversely, a homomorphism $\G^D \to \Aut(\g)$ gives $\g$ the structure of a comodule under the coordinate ring of $\G^D$, i.e., the group ring $R[\G]$, which amounts to an $R$-linear map $\rho \!: A \to A \ot R[\G]$ that is compatible with the Hopf algebra structure on $R[\G]$.  One checks that $\g_{\gamma} := \{ y \in \g \mid \rho(y) = y \ot \gamma \}$ gives a $\G$-grading on $\g$.\end{proof}

As is clear from the proof, the hypothesis that $\g$ is a Lie algebra is not necessary, and the result applies equally well to other sorts of algebraic structures over $R$.  One could also replace $R$ by a base scheme.

\subsection*{Examples: gradings by free abelian groups} By the proposition, a grading of $\g$ by a free abelian group $\Z^r$ corresponds to a homomorphism of a rank $r$ torus $\Gm^r \to \Aut(\g)$.  The image is connected, so it lies in the identity component of $\Aut(\g)$, which is the adjoint group, call it $\Gad$, with Lie algebra $\g$.  Up to conjugacy, all homomorphisms $\Gm^r \to \Gad$ have image in a given maximal torus $T$, and we conclude that all gradings of $\g$ by $\Z^r$ are obtained from a  rank $r$ sublattices of $\Hom(\Gm, T)$, i.e., the weight lattice for the root system dual to $R$.  ($\E_8$ is self-dual and so in that case it is the same as $Q$.)

A natural grading of this sort is obtained by choosing some $r$-element subset $\D'$ of $\D$ and defining the grade of a root vector $x_\beta$ to be the coefficients $c_\delta$ of elements of $\D'$ in the expression $\beta = \sum_{\delta \in \D} c_\delta \delta$.  Basic facts concerning this kind of grading can be found in \cite{ABS} or \cite{Rubenthaler}; most importantly, the $\g_\gamma$ are irreducible representations of $\g_0$ with an open orbit.  This sort of grading on $\e_8$ (and $\e_7$ and $\e_6$) with $r = 2$ was used in \cite{BDeMedts:new} to reconstruct the quadrangular algebras studied in \cite{Weiss:quad} for studying the corresponding Moufang polygons.

The simplest example of a grading as in the previous paragraph is to take $\D'$ to be a singleton $\{ \delta' \}$, in which case one finds a $\Z$-grading of $\g$ with support $\{ -n, \ldots, n \}$ where $n$ is the coefficient of $\delta'$ in the highest root.  Famously, in the case $n = 1$, one finds a 3-term grading in which the $\pm 1$ components make up a Jordan pair, and one can use the Jordan pair to recover the bracket on $\g$, see for example \cite{Neher:3}.  But that case does not occur for $\E_8$, because its highest root is 
\begin{equation} \label{highest}
\omega_8 = e_7 + e_8 = 2\alpha_1 + 3\alpha_2 + 4\alpha_3 + 6\alpha_4 + 5\alpha_5 + 4\alpha_6 + 3\alpha_7 + 2\alpha_8,
\end{equation}
where the smallest coefficient is 2 for $\delta' = \alpha_1$ or $\alpha_8$.  Those $\delta'$ give a 5-term grading, where the $\pm 1$ components are structurable algebras as in \cite{A:survey} or \cite[Th.~4]{A:models} or a Kantor pair as in \cite{AFS:Weyl}, depending on your point of view.  Starting with a structurable algebra, one can also recover the bracket on $\g$, as was done for $\e_8$ in \cite[\S8]{A:models}.

\subsection*{Examples: $\Z/n$ gradings}
By the proposition, a $\Z/n$ grading on $\g$ is given by a homomorphism $\mu_n \to \Aut(\g)$.  The ones with image in $\Gad$ (as is necessarily true for $\E_8$, because $\Aut(\g) = \Gad$ in that case) are classified by their Kac coordinates, see \cite{Kac:inf} or \cite{Se:Kac}, or see \cite{Lehalleur} for an expository treatment.  Therefore in principle one has a list of all of the corresponding gradings.

A popular choice for such a grading is to start by picking a simple root $\delta' \in \D$, thus obtaining as in the preceding subsection an inclusion $\iota \!: \Gm \to \Gad$ such that, for $\delta \in \D$, the composition $\delta  \iota \!: \Gm \to \Gm$ is the identity for $\delta = \delta'$ and trivial otherwise.  The desired $\Z/n$ grading is obtained by composing $\iota$ with the natural inclusion $\mu_n \hookrightarrow \Gm$ for $n$ the coefficient of $\delta'$ in the highest root.  One finds again in this case that each $\g_\gamma$ is an irreducible representation of $\g_0$, see \cite{Vinberg:Weyl}. The identity component of the centralizer of $C_{\Gad}(\mu_n)$ is semisimple and its Dynkin diagram is obtained by taking the so-called \emph{extended Dynkin diagram} of $R$ --- obtained by adding the negative of the highest root as a new vertex --- and deleting the vertex $\delta'$ and all edges connected to it.  (This is the first step in the iterative procedure described in \cite{BoSieb} for calculating the  maximal-rank  semisimple subgroups of $\Gad$.  See \cite[Ch.~VI, \S4, Exercise 4]{Bou:g4}  or \cite{Lehalleur} for treatments in the language of root systems or group schemes, respectively.)

For example, doing this for $R$ the root system of type $\mathsf{G}_2$ and $\delta'$ the simple root that is orthogonal to the highest root, one finds $n = 3$, $C_{\Gad}(\mu_3)$ has identity component $\SL_3$, and the $\pm 1$-components of $\g_2$ are the tautological representation and its dual.  The action of $G_2$ on the octonions, restricted to this $\mu_3$ subgroup, gives the direct sum decomposition of the octonions known as the Zorn vector matrix construction as in \cite{Schfr}.

\begin{eg}[$D_8 \subset E_8$] \label{D8}
Doing this with $E_8$ and $\delta' = \alpha_1$, we find $n = 2$ by \eqref{highest}, $C_{E_8}(\mu_2)$ has identity component $\Spin_{16}/\mu_2$, and the 1-component of $\e_8$ is the half-spin representation.  This grading was surely known to Cartan, and is the method used to construct $\e_8$ in  \cite{Adams:ex} and \cite{Fig}.
\end{eg}

\begin{eg}[$A_8 \subset E_8$] \label{A8E8}
Doing this with $E_8$ and $\delta' = \alpha_2$, we find $n = 3$, $C_{E_8}(\mu_3)$ has identity component $\SL_9/\mu_3$, and we have 
\[
\e_8 \cong \sl_9 \oplus (\wedge^3 \C^9) \oplus (\wedge^3 \C^9)^*
\]
as modules under $\sl_9$.  One can use this operation to define the Lie bracket on all of $\e_8$, see for example \cite{Frd:E8}, \cite[Exercise 22.21]{FH}, and \cite{Faulk:E8}.
\end{eg}

\begin{eg}[$A_4 \times A_4 \subset E_8$] \label{A4A4}
Taking $\delta' = \alpha_5$, we find $n = 5$, and $C_{E_8}(\mu_5)$ is isomorphic to $(\SL_5 \times \SL_5)/\mu_5$.  See \cite[\S14]{G:lens} or \cite[\S6]{GiQueg} for concrete descriptions of the embedding, or \cite[1.4.2]{Se:sgf} for how it provides a non-toral elementary abelian subgroup of order $5^3$.
\end{eg}

Dynkin \cite[Table 11]{Dynk:ssub} lists 75 different subsystems of $E_8$ constructed in this way, and each of these gives a way to construct the Lie algebra $E_8$ from a smaller one.  Here is such an example:
\begin{eg}[$D_4 \times D_4 \subset E_8$] \label{D4D4}
There is a subgroup $\mu_2 \times \mu_2$ of $E_8$ such that each of the three involutions has centralizer $\Spin_{16}/\mu_2$ as in Example \ref{D8}, and the centralizer of $\mu_2 \times \mu_2$ is the quotient of $\Spin_8 \times \Spin_8$ by a diagonally embedded copy of $\mu_2 \times \mu_2$.  This gives a grading of $\e_8$ by the Klein four-group, with $(\e_8)_0 = \so_8 \times \so_8$ and the other three homogeneous components being tensor products of an 8-dimensional representation of each of the copies of $\so_8$.  The corresponding construction of $E_8$ is the one described in \cite[Ch.~IV]{Walde}, \cite{Loke:th}, or \cite{LM:tri}.
\end{eg}

We have described some large families of gradings, but there are of course others, see for example \cite{eld:Jordan}.  One can also construct $\e_8$ without using a grading.  The famous \emph{Tits construction} of $E_8$ is of this type, and we will describe it in detail in \S\ref{Tits.sec} below.

\section{$E_8$ over the real numbers}

A Lie algebra $L$ over the real numbers, i.e., a vector space over $\R$ with a bracket that satisfies the axioms of a Lie algebra, is said to be of type $E_8$ if $L \ot_\R \C$ is isomorphic to the Lie $\C$-algebra $\e_8$ studied by Killing.  (The notation $L \ot_\R \C$ is easy to understand concretely.  Fix a basis $\ell_1, \ldots, \ell_n$ for the real vector space $L$.  Then $L \ot_\R \C$ is a complex vector space with basis $\ell_1, \ldots, \ell_n$ and the bracket on $L \ot_\R \C$ is defined by $[\sum_i z_i \ell_i, \sum_j z'_j \ell_j] = \sum_i \sum_j z_i z'_j [\ell_i, \ell_j]$ for $z_i, z'_j \in \C$.)

Cartan proved in \cite{Cartan:real} that there are exactly three Lie $\R$-algebras of type $E_8$.  One can reduce the problem of determining the isomorphism class of such algebras to determining the orbits under the Weyl group of elements of order 2 in a maximal torus in the compact real form of $E_8$, see \cite[\S{III.4.5}]{SeCG}.  Given such an involution and the description of $E_8$ in terms of the Chevalley relations, one can do calculations, so this is an effective construction.  Two of the three $E_8$'s over $\R$ have names: split and compact, where the automorphism group of the compact Lie algebra is a real Lie group that is compact.  The third one does not have a standard name; we call it \emph{intermediate} following \cite{Selet2}, but others say \emph{quaternionic}.  This is summarized in Table \ref{real.table}. The first column of the table indicates the Tits index or Satake diagram for the group; the number of circles gives the dimension of a maximal split torus.
\begin{table}[hbt]
\begin{tabular}{ccr} 
Tits index&&signature of \\ 
or Satake diagram&name&Killing form\\ \hline
\begin{picture}(90,21)(0,4)
 \put(0,0){\usebox{\Eviii}} 
\end{picture}& compact & $-248$\\
\begin{picture}(90,21)(0,4)
 \put(0,0){\usebox{\Eviii}} 
 \put(0,6){\circle{7}}
 \multiput(60,6)(15,0){3}{\circle{7}}
\end{picture}& intermediate&$-24$\\
\begin{picture}(90,21)(0,4)
 \put(0,0){\usebox{\Eviii}} 
 \put(30,21){\circle{7}}
 \multiput(0,6)(15,0){7}{\circle{7}}
\end{picture}&split&8
\end{tabular}
\caption{Real forms of $E_8$} \label{real.table}
\end{table}

\subsection*{Physics}
Various forms of $E_8$ play a role in physics.  For example, the compact real form appears in string theory --- see \cite{GHMR}, \cite{DMW}, and \cite{DMF} --- and the split real form appears in supergravity \cite{MarcusSchwarz}.

On the other hand, 
$E_8$ does not play any role in the Grand Unified Theories of the kind described in \cite{BaezHuerta} because every finite dimensional representation of every real form of $E_8$ is real orthogonal and not unitary.  And at least one well-publicized approach to unification based on $E_8$ does not work, as explained in \cite{DiG}. 

\subsection*{A magnet detects $E_8$ symmetry} 
A few years ago, experimental physicsts reported finding ``evidence for $E_8$ symmetry'' in a laboratory experiment \cite{Coldea}.  This is yet a different role for $E_8$ than those mentioned before.  The context is as follows.  A one-dimensional magnet subjected to an external magnetic field can be described with a quantum one-dimensional Ising model --- this is standard --- and, in the previous millennium and not using $E_8$, \cite{Zam} and \cite{DelfinoMussardo} made numerical predictions regarding what would happen if the magnet were subjected to a second, orthogonal magnetic field.  In the intervening years, such an experiment became possible.  Some of the difficulties that had to be overcome were producing a large enough crystal with the correct properties (such as containing relatively isolated one-dimensional chains of ions) and producing a magnetic field that was large enough (5 tesla) and tunable.

The measurements from the experiment were as predicted, and therefore the experiment can be viewed as providing evidence for the theory underlying the predictions.\footnote{Some skeptics aver that the evidence is weak, in that the confirmed predictions  amount to somewhat noisy measurements of two numbers, namely the ratio of the masses of the two heaviest quasi-particles and the ratio of the corresponding intensities.  In Bayesian language, the skeptics' estimate of the (probability of) correctness of the underlying theory was only somewhat increased by learning of the experimental results.}  Which raises the question: how does $E_8$ appear in these predictions?  Indeed, looking at the original papers, one finds no use of $E_8$ in making the predictions (although Zamolodchikov did remark already in \cite{Zam} that numerological coincidences strongly suggest that $E_8$ should play a role).  Those papers, however, in addition to a handful of unobjectionable non-triviality assumptions, rest on a serious assumption that the perturbed conformal field theory would be an integrable theory. It was later discovered that if one views the original conformal field theory as arising from a compact real Lie algebra via the ``coset construction'', then the perturbed theory is given by the corresponding affine Toda field theory, see \cite{FateevZam}, \cite{HollowoodMansfield}, and \cite{EguchiYang}.  In this way, invoking the compact real form of $E_8$ simplifies the theory by removing an assumption, so evidence for the theory may be viewed as ``evidence for $E_8$ symmetry''. See \cite{BoG} or \cite{Kost:gosset} for more discussion.

A further experiment to try would be to test the predictions for a thermal perturbation of the tricritical Ising model, for which the role of $E_8$ is played instead by $E_7$.  The unperturbed version has already been realized \cite{TFV}.

\section{$E_8$ over any field} \label{fields.sec}

It makes sense to speak of Lie algebras or groups ``of type $E_8$'' over any field.  
Starting with the $E_8$ root system and the Chevalley relations as in the previous section, one can put a Lie algebra structure on $\Z^{248}$.  The isomorphism class of the resulting Lie algebra $\e_{8,\Z}$ does not depend on the choice of signs, and in this way one obtains, for every field $F$, a Lie $F$-algebra $\e_{8,F} := \e_{8,\Z} \ot F$ called the \emph{split $\e_8$ over $F$}.  Its automorphism group $E_{8,F}$ is an algebraic group over $F$\footnote{Meaning it is a smooth variety such that the functor $\{ \text{field extensions of $F$} \} \to \Sets$ given by sending $K \mapsto E_{8,F}(K)$ factors through the category of groups.} known as the \emph{split group of type $E_8$ over $F$}.

\subsection*{Finite fields}
For example, if $F$ is a \emph{finite} field with $q$ elements, then the $F$-points $E_{8,F}(F)$ are a finite simple group, normally denoted $E_8(q)$.  These are enormous groups, whose order is a polynomial in $q$ with leading term $q^{248}$.  The smallest such group is $E_8(2)$, which has more than $3 \times 10^{74}$ elements, about $10^{20}$ times the size of the Monster \cite[p.~242]{Atlas}.   A special case of the Inverse Galois Problem is to ask: Does $E_8(q)$ occur as a Galois group over $\Q$?  It is known when $q$ is a prime larger than 5 by \cite{GurMalle:E8} and \cite{Yun}.

\subsection*{Groups of type $E_8$}
More generally, a Lie algebra $L$ over a field $F$ is said to \emph{have type $E_8$} or be an \emph{$F$-form of $E_8$} if there is a field $K \supseteq F$ such that $L \ot_F K$ is isomorphic to $L_0 \ot_\Z K$.  A similar definition applies for algebraic groups.

An algebraic group $G$ of type $E_8$ over $F$ is \emph{split} if it is isomorphic to $E_{8,F}$.  It is \emph{isotropic} if it contains a copy of $\Gm$ (equivalently, if $\Lie(G)$ has a nontrivial $\Z$-grading), and \emph{anisotropic} otherwise.  For $F = \R$, there is a unique anisotropic form of $E_8$, the compact form, but for other fields $F$ there may be many, see the examples in \S\ref{rost.sec}.

\subsection*{Fields over which the split $E_8$ is the only $E_8$}
Over some fields $F$, the split group $E_{8,F}$ is the only group of type $E_8$.  This is true for algebraically closed fields (trivially by \cite{SGA3}).  It is also true for finite fields, fields of transcendence degree 1 over an algebraically closed field, and more generally for fields of cohomological dimension at most 1 \cite{St:reg}.  Classically it is also known for $p$-adic fields and it is also true for every field that is complete with respect to a discrete valuation with residue field of cohomological dimension $\le 1$ \cite{BrTi3}.  It is also true for global fields of prime characteristic, i.e., a finite extensions of $\F_p(t)$ for some $p$ \cite{Hrdr:3}.  

\subsection*{Serre's ``Conjecture II''} A bold conjecture of Serre's from 1962 \cite{Se62} says that, for every field $F$ of cohomological dimension\footnote{Here \emph{cohomological dimension} means the supremeum of the cohomological $p$-dimension as $p$ varies over all primes. Cohomological $p$-dimension is defined via Galois cohomology as in \cite{SeCG} for $p \ne \car F$ and via Milne-Kato cohomology if $p = \car F$.  Furthermore, the conjecture only demands this for a finite list of primes depending on the Killing-Cartan type of $G$, see \cite{SeCGp} or \cite{Gille:scsurv} for precise statements.  In the case of $E_8$, the hypothesis only concerns the primes 2, 3, and 5.} $\le 2$, and $G$ a simply connected semisimple algebraic group over $F$, $H^1(F, G) = 0$.  As $E_{8,F} = \Aut(E_{8,F})$, by descent $H^1(F, E_{8,F})$ is identified with the pointed set of groups of type $E_8$ over $F$ (with distinguished element $E_{8,F}$ itself).  In this case, Serre's conjecture says: the split group $E_{8,F}$ is the only form of $E_8$ over $F$.  What made this conjecture bold is that it includes totally imaginary number fields and finite extensions of $\F_p(t)$ as special cases and the conjecture was not proved in these special cases until the 1980s \cite{Ch:hasse} and 1970s \cite{Hrdr:3} respectively.  Furthermore, the hypothesis on cohomological dimension is very difficult to make use of; a key breakthrough in its application was the Merkurjev--Suslin Theorem \cite[Th.~24.8]{Suslin:norm} connecting it with surjectivity of the reduced norm of a central simple associative algebra.

Despite lots of progress during the past 20 years as in \cite{BP}, and \cite{Gille:sc}, the conjecture remains open; see \cite{Gille:scsurv} for a recent survey.  It is known to hold for $E_8$ if:
\begin{itemize}
\item every finite extension of $k$ has degree a power of $p$ for some fixed $p$ by \cite[\S{III.2}]{Gille:sc} for $p = 2, 3$ (or see \cite{Ch:mod3} or \cite[10.22]{GPS} for $p = 3$) and by \cite{Ch:mod5} for $p = 5$ (or see \cite[15.5]{G:lens}).  The $p = 2$ case can also be deduced from properties of Semenov's invariant (see Theorem \ref{semenov}) if $\car F = 0$.
\item $k$ is the function field of a surface over an algebraically closed field, see \cite{dJHS} and \cite{Gille:scsurv}.
\end{itemize}

\subsection*{Forms of $E_8$ over $\Z$} One can equally well consider $\Z$-forms of $E_8$ as in \cite{Gross:Z} or \cite{Conrad:Z}.  We don't go into this here, but note that $E_8$ is unusual in this context in that the number of $\Z$-forms of the ``compact'' real $E_8$ is at least $10^4$, in stark contrast to at most 4 for other simple groups of rank $\le 8$, see \cite[Prop.~5.3]{Gross:Z}.

\section{Tits's construction} \label{Tits.sec}

Tits gave an explicit construction of Lie algebras which produces $E_8$ as a possible output \cite{Ti:const}.
Computing using root systems, one can see that the Lie algebra $\e_8$ contains $\g_2 \times \f_4$ as a subalgebra. It corresponds to an inclusion of groups $G_2 \times F_4 \injects E_8$.  Applying Galois cohomology $H^1(F, *)$, we find a map of pointed sets
\begin{equation} \label{Tits.map}
H^1(F, G_2) \times H^1(F, F_4) \to H^1(F, E_8)
\end{equation}
that is functorial in $F$; we call this \emph{Tits's construction}.  To interpret this, we note that $E_8$ is its own automorphism group and similarly for $F_4$ and $G_2$.  Again, descent tells us that $H^1(F, E_8)$ is naturally identified with the set of isomorphism classes of algebraic groups (or Lie algebras) of type $E_8$ over $F$, with distinguished element the split group $E_8$.  The same holds for $F_4$ and $G_2$.  Thus the function \eqref{Tits.map} can be viewed as a construction that takes as inputs a group of type $G_2$ and a group of type $F_4$ and gives output a group (or Lie algebra) of type $E_8$.

Alternatively, viewing $G_2$ and $F_4$ as automorphism groups of an octonion and an Albert algebra respectively, Galois descent identifies $H^1(F, G_2)$ and $H^1(F, F_4)$ with the sets of isomorphism classes of octonion $F$-algebras and Albert $F$-algebras respectively, cf.~\cite{KMRT}.  From this perspective, the construction takes as input two algebras.

Tits expressed his construction as an explicit description of a Lie algebra $L$ made from an octonion $F$-algebra and an Albert $F$-algebra, see \cite{Ti:const} or \cite{Jac:ex} for formulas.

\begin{eg} \label{Tits.rkill}
Let $G$ be a group of type $E_8$ obtained from an octonion algebra $C$ and Albert algebra $A$ via Tits's construction; we give the formula for the isomorphism class of its Killing form.  The algebra $C$ is specified by a 3-Pfister quadratic form $\gamma_3$ and $A$ has $i$-Pfister quadratic forms $\phi_i$ for $i  = 3, 5$, see \cite{GMS}.  Formulas for the Killing forms on the subgroups $\Aut(C)$ and $\Aut(A)$ are given in \cite[27.20]{GMS}, and plugging these into the calculation in \cite[p.~117, (144)]{Jac:ex} gives that the Killing form on $G$ is
\[
\qform{60}[8 - (4 \gamma_3 + 4 \phi_3 + \qform{2} \gamma_3(\phi_5 - \phi_3) )]
\]
in the Witt ring of $k$, exploiting the notation for elements of the Witt ring from \cite{EKM}.
\end{eg}

In case $F = \R$, we know that there are two groups of type $G_2$ or octonion algebras, the compact form (corresponding to the octonion division algebra) and the split form (the split algebra).  There are three real groups of type $F_4$ or Albert $\R$-algebras, the compact form (the Albert algebra with no nilpotents), the split form (the split algebra), and a third form we will call ``intermediate'' (the Albert algebra with nilpotents but not split).  The book \cite{Jac:ex} contains a careful calculation of the Killing form on the Lie algebra of type $\E_8$ produced by this construction for each choice of inputs, and we list the outputs in Table \ref{real.table2}.

\begin{table}[hbt]
\begin{tabular}{c|cc|cc}
&Rost&Semenov&\multicolumn{2}{c}{inputs to Tits's construction} \\
Form of $E_8$&invariant&invariant&$G_2$&$F_4$\\ \hline\hline
compact&0&1&compact&compact \\ \hline
intermediate &1&NA&split&intermediate or compact\\
&&&compact&split \\ \hline
split&0&0&split&split \\
&&&compact&intermediate
\end{tabular}
\caption{Real forms of $E_8$: cohomological invariants and Tits's construction} \label{real.table2}
\end{table}

One could ask for a description of the Tits index of the $E_8$ produced by \eqref{Tits.map} in terms of the the inputs, e.g., to know whether and how the output group is isotropic.  In one direction, the answer is easy: 
 the groups of type $G_2$ and $F_4$ used as inputs to the construction are subgroups of the output, and therefore isotropic inputs give isotropic outputs.  However, in the analogous question for Tits's construction of groups of type $F_4$ or $E_6$, anisotropic inputs can produce isotropic outputs, see for example \cite{GPe}.

\begin{eg}[Number fields] \label{Hasse}
If $G$ is a group of type $E_8$ over a number field $F$, then the natural map
\[
H^1(F, G) \to \prod_{\text{real closures $R$ of $F$}} H^1(R, G)
\]
is bijective.  (This is true more generally for every semisimple simply connected group $G$ \cite[p.~286]{PlatRap}.  It is the Hasse Principle mentioned in the introduction.)  In particular, if $F$ has $r$ real embeddings, then there are $3^r$ isomorphism classes of groups of type $E_8$ over $F$, all of which are obtained from Tits's construction.  Two groups of type $E_8$ over $F$ are isomorphic if and only if their Killing forms are isomorphic, because this is so when $F = \R$.
\end{eg}

\section{Cohomological invariants; the Rost invariant} \label{rost.sec}

At this point in the survey, we have assembled only a few tools for determining whether or not two groups or Lie algebras of type $E_8$ over a field $F$ are isomorphic, or whether they can be obtained as outputs from Tits's construction.  Indeed, the only tools we have so far for this are the Tits index and the Killing form.  In \cite{Ti:deg}, Tits showed that a ``generic''\footnote{Meaning a \emph{versal} form as defined in \cite{GMS}.  By definition, one can obtain from a given versal form of $E_8$ every group of type $E_8$ over every extension of $F$ by specialization.}
 group $G$ of type $E_8$ has no reductive subgroups other than rank 8 tori (which necessarily exist); such a group cannot arise from Tit's construction because those groups all have a subgroup of type $G_2 \times F_4$.
 This brings us to the state of knowledge in 1990, as captured in \cite{Se9091}.  
 
To address this gap, we will use various invariants in the sense of \cite{GMS}.  Let $G$ be an algebraic group over a field $F$.  Let $A$ be a functor from the category of fields $F$ to abelian groups.  An \emph{$A$-invariant $f$ of $G$} is a morphism of functors $f \!: H^1(*, G) \to A(*)$, i.e., for each field $K$ containing $F$, there is a function $f_K \!: H^1(K, G) \to A(K)$ and for each morphism $\alpha \!: K \to K'$ the diagram
\[
\begin{CD}
H^1(K, G) @>{f_K}>> A(K) \\
@V{H^1(\alpha)}VV @VV{A(\alpha)}V \\
H^1(K', G) @>{f_{K'}}>>A(K')
\end{CD}
\]
commutes.  We write $\Inv(G, A)$ for the collection of such invariants; it is an abelian group.  We put $\Inv(G, A)_0$ for the subgroup of those $f$ such that $f_K$ sends the distinguished element to zero for all $K$.  Evidently $\Inv(G, A) = \Inv(G, A)_0 \oplus A(F)$.

\subsection*{The Rost invariant}
Considering the case where $A$ is Galois cohomology with torsion coefficients and $G$ is a simple and simply connected group, one finds that $\Inv(G, H^d(*, \Q/\Z(d-1))$ is zero for $d = 1$ (because $G$ is connected \cite[31.15]{KMRT}) and for $d = 2$ (because $G$ is simply connected, see \cite[Th.~2.4]{BlinMer}).  For $d = 3$, Markus Rost proved that the group is cyclic with a canonical generator $r_G$, now known as the \emph{Rost invariant}, whose basic properties are developed in \cite{GMS}.  (Rost's theorem has recently been extending to include the case where $G$ is split reductive, see \cite{M:ssinv2} and \cite{LaackmanM}.)  For $G = G_2, F_4, E_8$, $r_G$ has order 2, 6, 60 respectively.  

\begin{eg}
The composition 
\[
H^1(F, G_2) \times H^1(F, F_4) \xrightarrow{\mathrm{Tits}} H^1(F, E_8) \xrightarrow{r_{E_8}} H^3(F, \Q/\Z(2))
\]
is $r_{G_2} + r_{F_4}$, by an argument analogous to that in the proof of \cite[Lemma 5.8]{GQ}.  In particular, the image belongs to the 6-torsion subgroup, $H^3(F, \Z/6(2))$.  Using this, one uses the (known, calculable) Rost invariants $r_{G_2}$ and $r_{F_4}$ to calculate $r_{E_8}$ for all forms of $E_8$ over $\R$ or a number field.  We have included these values in Table \ref{real.table2}.
\end{eg} 

\begin{eg} \label{rost.mod5}
Considering the $\Z/5$-grading on $\e_8$ from Example \ref{A4A4}, we find a subgroup $(\SL_5 \times \SL_5)/\mu_5$ of $E_8$, which in turn contains a subgroup $\mu_5 \times \mu_5 \times \Z/5$.  The composition
\begin{equation} \label{rost.5}
H^1(F, \mu_5) \times H^1(F, \mu_5) \times H^1(F, \Z/5) \to H^1(F, E_8) \xrightarrow{r_{E_8}} H^3(F, \Q/\Z(2))
\end{equation}
is, up to sign, the cup product map into $H^3(F, \Z/5(2))$ \cite{GiQueg}.  Therefore, if $H^3(F, \Z/5(2))$ is not zero, it contains a nonzero symbol $s$, and from this one can construct a group $G$ of type $E_8$ with $r_{E_8}(G) = s$.  This group $G$ cannot arise from Tits's construction.

If $F$ is 5-special --- meaning that every finite extension of $F$ has degree a power of 5 --- then the first map in \eqref{rost.5} is surjective.
\end{eg}

\begin{eg} \label{rost.mod3}
We can focus our attention on the 3-torsion part of the image of $r_{E_8}$ by considering $20 r_{E_8}$, which has image in $H^3(*, \Z/3(2))$.  What is its image?  Given a nonzero symbol $s$ in $H^3(F, \Z/3(2))$, there is an Albert division algebra $A$ constructed by the first Tits construction such that $r_{F_4}(A) = s$.  The inclusions $F_4 \subset E_6 \subset E_8$ induce a map $H^1(F, F_4) \to H^1(F, E_8)$ such that the image $G$ of $A$ has Tits index
\begin{equation} \label{E8E6}
\begin{picture}(90,21)(0,4)
 \put(0,0){\usebox{\Eviii}} 
 \multiput(75,6)(15,0){2}{\circle{7}}
\end{picture}
\end{equation}
and $r_{E_8}(G) = r_{F_4}(A) = s$.  Conversely, given an $F$-form $G$ of $E_8$ over any field $F$ such that $20r_{E_8}(G)$ is a symbol in $H^3(F, \Z/3(2))$, \cite{GPS} shows that there is a finite extension $K$ of $F$ of degree not divisible by 3 such that $G \times K$ has Tits index \eqref{E8E6}.
\end{eg}

\begin{eg} \label{mod3.symbol}
There exists a field $F$ and an $F$-form $G$ of $E_8$ such that $20 r_{E_8}(G)$ is \emph{not} a symbol in $H^3(F, \Z/3(2))$, nor does it become a symbol over any finite extension of $F$ of degree not divisible by 3.  (In particular, such a group does not arise from Tits's construction.)  To see this, one appeals to the theory of the $J$-invariant from \cite{PSZ:J} that there exists such a $G$ with $J_3(G) = (1, 1)$; \cite[Lemma 10.23]{GPS} shows that such a group has the desired property.
\end{eg}

In contrast to the case of 3 and 5 torsion as in the preceding two examples, the 2-primary part of $r_{E_8}(G)$ may have a longer symbol length and a more subtle relationship with the isotropy of $G$.  As in Table \ref{real.table2}, the compact real form has Rost invariant 0, yet is anisotropic.  In the other direction,
Appendix A.6 of \cite{G:lens} gives an example of an isotropic group $G$ of type $E_8$ over a field $F$ such that $r_{E_8}(G)$ belongs to $H^3(F, \Z/2(2))$, and $r_{E_8}(G)$ is not a sum of fewer than three symbols.

Underlying the results used in Examples \ref{rost.mod3} and \ref{mod3.symbol} are analyses of the possible decompositions of the Chow motives with $\F_p$-coefficients of the flag varieties for a group of type $E_8$.  This is a way to study the geometry of these varieties despite their lack of rational points in cases of interest, and it has been widely exploited over the last decade for analyzing semisimple groups over general fields, see for example the book \cite{EKM} for applications to quadratic form theory.

\begin{eg}
Suppose $G$ is a group of type $E_8$ such that $r_{E_8}(G)$ has order divisible by 15.  (For example, if $G$ is versal then $r_{E_8}(G)$ has order 60 by \cite[p.~150]{GMS}.)  Then $G$ has no proper reductive subgroups other than rank 8 tori, by the argument for \cite[Th.~9.6]{GaGi}.  Essentially, employing the Rost invariant simplifies Tits's proof of the fact about versal groups mentioned at the beginning of this section.
\end{eg}

\section{The kernel of the Rost invariant; Semenov's invariant}

The examples in the previous section show that the Rost invariant can be used to distinguish groups of type $E_8$ from each other and for stating some results about such groups.  But it remains a coarse tool.  For example: \emph{Given a field $F$, what are the groups of type $E_8$ that are in the kernel of the Rost invariant?}  For this we have little to offer beyond the special fields discussed in \S\ref{fields.sec} and the following.

\begin{thm}[Tits-Chernousov]
Suppose there is an \underline{odd} prime $p$ such that every finite extension of $F$ has degree a power of $p$.  Then the kernel of the Rost invariant $r_{E_8} \!: H^1(F, E_8) \to H^3(F, \Q/\Z(2))$ is zero.
\end{thm}

\begin{proof}
This was observed for $p \ge 7$ in \cite{Ti:deg}, with only the $p = 7$ case being nontrivial.    The cases $p = 5, 3$ are the main results of \cite{Ch:mod5} and \cite{Ch:mod3}, see also \cite[Prop.~15.5]{G:lens} for $p = 5$ and \cite[Prop.~10.22]{GPS} for $p = 3$.
\end{proof}

The prime 2 is omitted in the theorem because the conclusion is false in that case, as can be seen already for $F = \R$, where the compact $E_8$ has Rost invariant zero.  Recently, Semenov produced a newer, finer invariant in \cite{Sem:mot}.  Put $H^1(F, E_8)_{15}$ for the kernel of $15r_{E_8}$, a subset of $H^1(F, E_8)$.

\begin{thm}[Semenov] \label{semenov}
If $\car F = 0$, there exists a nonzero invariant
\[
s_{E_8} \!: H^1(*, E_8)_{15} \to H^5(*, \Z/2)
\]
such that $s_{E_8}(G) = 0$ if and only if $G$ is split by an odd-degree extension of $F$.
\end{thm}

For the case $F = \R$, $s_{E_8}$ is defined only for the split and compact real forms, and it is nonzero on the compact form.

This style of invariant, which is defined only on a kernel of another invariant, is what one finds for example in quadratic form theory.  The Rost invariant $H^1(*, \Spin_n) \to H^3(F, \Z/2)$ amounts to the so-called Arason invariant $e_3$ of quadratic forms.  There is an invariant $e_4$ defined on the kernel of $e_3$ with values in $H^4(F, \Z/2)$, an invariant $e_5$ defined on the kernel of $e_4$ with values in $H^5(F, \Z/2)$, and so on, see \cite{EKM}. 

\begin{eg} 
Let $G$ be a group of type $E_8$ arising from Tits's construction as in Example \ref{Tits.rkill}.  Then $15r_{E_8}(G) = 0$ if and only if $\phi_3 = \gamma_3$ in the notation of that example.  In that case, $s_{E_8}(G)$ is the class  $e_5(\phi_5) \in H^5(F, \Z/2)$, see \cite[Th.~3.10]{GS:deg5}.
\end{eg}

It is natural to wonder what other cohomological invariants of $E_8$ may exist. 

\begin{ques} \cite[p.~1047]{RY} Do there exist nonzero  invariants mapping $H^1(*, E_8)$ into $H^5(*, \Z/3)$ and $H^9(*, \Z/2)$?
\end{ques}

The parameters in the question are based on the existence of certain non-toral elementary abelian subgroups in $E_8(\C)$.  (Such subgroups are described in \cite{Griess:el}.)

\section{Witt invariants}

We introduced the notion of $A$-invariant as a tool for distinguishing groups of type $E_8$.  In the previous two sections we discussed the case where $A$ was a Galois cohomology group, but one could equally well take $A$ to be the functor $W$ that sends a field $F$ to its Witt group of non-degenerate quadratic forms $W(F)$ as defined in, for example, \cite{EKM} or \cite{Lam}.  In this section, for ease of exposition, we assume that $\car F \ne 2$.  In particular, $W(F)$ is naturally a ring via the tensor product.

\subsection*{Killing form}
For $G$ an algebraic group, put $\kappa(G)$ for the Killing form on the Lie algebra of $G$.  (If $\car k$ divides 60, then the Killing form is identically zero, and in that case one should define $\kappa(G)$ to be the reduced Killing quadratic form as defined in \cite{G:vanish}.  We elide the details here.) 
The function $G \mapsto \kappa(G)$ is a morphism of functors $H^1(*, E_8) \to W(*)$, and we define
\[
f \!: H^1(*, E_8) \to W(*) \quad \text{via} \quad G \mapsto \kappa(G) - \kappa(E_8)
\]
to obtain a $W$-invariant that sends the distinguished class to 0.

The Witt ring has a natural filtration by powers of the ``fundamental'' ideal $I$ generated by even-dimensional quadratic forms.  It is natural to ask: \emph{What is the largest $n$ such that $f(G)$ belongs to $I^n$?}  

\begin{eg}
In case $G$ arises by Tits's construction, the formula in Example \ref{Tits.rkill} shows that $f(G) \in I^5$.  Concretely we can see this over the real numbers, where the signature defines an isomorphism $W(\R) \cong \Z$ that identifies $I^5$ with $32\Z$, and Table \ref{real.table} shows that $f(G) \in \{ 0, -32, -248 \}$
\end{eg}

The action of the split group $E_8$ on its Lie algebra $\e_8$ is by Lie algebra automorphisms, so it preserves the Killing form $\kappa(E_8)$, giving a homomorphism $E_8 \to O(\kappa(E_8))$.  As $E_8$ is simply connected, this lifts to a homomorphism $E_8 \to \Spin(\kappa(E_8))$.  By Galois descent, it is clear that the image of a group $G$ under the map $H^1(F, E_8) \to H^1(F, O(\kappa(E_8))$ is $\kappa(G)$, whence $f(G)$ necessarily belongs to $I^3$ because the map factors through $H^1(F, \Spin(\kappa(E_8)))$ \cite[p.~437]{KMRT}.

In fact, we cannot do better than the exponent in the previous paragraph, because: 
\emph{$f(G)$ belongs to $I^4$ iff $30r_{E_8}(G) = 0$}.  To see this, we observe on the one hand that $f(G)$ belongs to $I^4$ if and only if $f(G)$ is in the kernel of the Arason invariant $e_3 \!: I^3(*) \to H^3(*, \Z/2)$ by \cite{MSus:K3} or \cite{Rost:H90}, i.e., the Rost invariant $r_{\Spin(\kappa(E_8))}$ vanishes on the image of $G$.  On the other hand, the map $E_8 \to \Spin(\kappa(E_8))$ has Dynkin index 30, so composing it with $r_{\Spin(\kappa(E_8))}$ gives $30r_{E_8}$ by \cite[p.~122]{GMS}.  This completes the proof of the claim.  A versal group $G$ of type $\E_8$ has $30r_{E_8}(G) \ne 0$, so $f(G) \in I^3 \setminus I^4$.

Inspired by \cite{Selet}, we may still ask: \emph{If $r_{E_8}(G)$ has odd order, is it necessarily true that $f(G) \in I^8$?}  If the answer is yes and $\car F = 0$, one can ask for more: \emph{Does Semenov's invariant $s(G) \in H^5(F, \Z/2)$ divide $e_8(f(G)) \in H^8(F, \Z/2)$?}  The answer to both of these questions is ``yes'' for groups arising from Tits's construction, by Example \ref{Tits.rkill}.

\subsection*{Witt invariants in general}
Here is a natural way to construct invariants with values in the Witt ring, i.e., morphisms of functors $H^1(*, E_8) \to W(*)$.  For each dominant weight $\la$ of $\E_8$, there is a Weyl module $V(\la)$ with highest weight $\la$ defined for $\e_{8,\Z}$, that has an indivisible, $E_{8,\Z}$-invariant quadratic form $q_\Z$ on it, uniquely determined up to sign.  Base change to a field $P$ gives a quadratic form $q_\Z \ot P$ on $V(\la) \ot P$ that has a radical, namely the unique maximal proper submodule of $V(\la) \ot P$, hence we find a quadratic form on the irreducible quotient of $V(\la) \ot P$, call the form $q_{\la,P}$.  Applying $H^1$, we obtain a function $H^1(P, E_8) \to H^1(P, O(q_{\la,P}))$, and therefore for every group $G$ of type $E_8$ over $P$, we obtain a corresponding quadratic form we denote by $q_\la(G)$.  For example, the ``constant'' invariant that sends every $G$ to the 1-dimensional quadratic form $x \mapsto x^2$ can be viewed as $q_0$, and the invariant $\kappa$ can be viewed as $q_{\omega_8}$.

\begin{ques}
\emph{For every field $K$ of characteristic different from $2$, $W(K)$ is a ring, so the set of invariant $H^1(*, E_8) \to W(*)$ is a ring.}  Do $W(F)$ and the $q_\la$, as $\la$ varies over the dominant weights, generate the ring of invariants?
\end{ques}


\section{Connection with division algebras}

Groups of type $E_8$ are closely related to the smallest open cases of two of the main outstanding problems in the study of (associative) division algebras.  To recall the terminology, the center of a division ring $D$ is a field, call it $F$.  We say that $D$ is a division algebra if $\dim_F D$ is finite, in which case $\dim_F D$ is a square, and its square root is called the \emph{degree} of $D$.  Hamilton's quaternions are an example of such a $D$ with $F = \R$ and $\deg D = 2$.

One knows that, among all fields $K$ such that $F \subset K \subset D$, the maximal ones always have $\dim_F K = \deg D$.  In case there exists a maximal $K$ that is Galois over $F$ (resp., Galois over $F$ with cyclic Galois group), one says that $D$ is \emph{a crossed product} (resp., is \emph{cyclic}).  For a crossed product, one can write down a basis and multiplication rules in a relatively compact way, and the description is even simpler if $D$ is cyclic.

Every division algebra $D$ with $\deg D = 2$ or $3$ is known to be cyclic, and the principal open problem in the theory of division algebras is: \emph{If $\deg D = p$ for a prime $p \ge 5$, must $D$ be cyclic?}  Philippe Gille showed that this question, for the case $p = 5$, can be rephrased as a statement about groups of type $E_8$, see \cite{Gille:E8}.

Along with the dimension, another property of a division algebra $D$ is its \emph{period}, which is the smallest number $p$ such that a tensor product of $p$ copies of $D$ is isomorphic to a matrix algebra over $F$.  (It is a basic fact that $p$ divides $\deg D$ and that the two numbers have the same prime factors.)  Another open problem about division algebras is: \emph{Determine whether every division algebra $F$ of period $p$ and degree $p^2$ is a crossed product.}  For period 2 and degree $2^2$, the answer is ``yes'' and is due to Adrian Albert.  (It is also ``yes'' for period $2$ and degree  $2^3$ by Louis Rowen \cite{Jac:fd}.)  The smallest open case, then, is where $p = 3$, and this is where $E_8$ may play a role.  As in \cite{Vinberg:tri}, using the $\Z/3$-grading on $\e_8$ from Example \ref{A8E8}, we find that $\SL_9/\mu_3$ acts transitively on certain 4-dimensional subspaces of $\wedge^3(F^9)$.  This gives a surjection in Galois cohomology $H^1(F, N) \to H^1(F, \SL_9/\mu_3)$ for some subgroup $N$ of $\SL_9/\mu_3$ and one can hope that analyzing this surjection would give insight into whether algebras of degree 9 and period 3 are crossed products.  See \cite{LoetscherMacD} for more discussion of this general setup and \cite{G:lens} for examples where similar surjections are exploited.

\section{Other recent results on $E_8$}

\subsection*{Torsion index} Grothendieck \cite{Groth:tor} defined an invariant of a compact Lie group $G$ called the \emph{torsion index}, which has interpretations for the Chow groups and motivic cohomology of the classifying space of $BG$ as well as for the \'etale cohomology of torsors under analogues of $G$ over arbitrary fields, see \cite{Tot:spin} for precise statements.  Jacques Tits's brief paper \cite{Ti:deg} contained a proof that the torsion index for the compact $E_8$ is divisible by $60 = 2 \cdot 3 \cdot 5$ and divides $2^9 \cdot 3^2 \cdot 5$, and wrote that one could hope (``hypoth\'ese optimiste !'') that the correct answer is 60.  In \cite{Tot:E8} and \cite{Tot:E8tor}, Burt Totaro proved that it is $2^6 \cdot 3^2 \cdot 5$.

\subsection*{Kazhdan-Lusztig-Vogan polynomials} Some readers will remember a flurry of news coverage in 2007 about a ``calculation the size of Manhattan''.  This referred to the calculation of the Kazhdan-Lusztig-Vogan polynomials for the split real forms of simple Lie algebras, where the final step was for $\e_{8,\R}$.  For more on this see \cite{Vogan} and \cite{Leeuwen:Vogan}.  We have also omitted the related topic of infinite-dimensional unitary representations of $E_8$, for constructions of which see example  \cite{KazhdanSavin}, \cite{GrossWallach}, and \cite{BrylinskiKostant}.

\subsection*{Finite simple subgroups} We now know which finite simple groups embed in the (infinite) simple group $E_8(\C)$, or more generally $E_8(k)$ for $k$ algebraically closed, see \cite{Se:sgf}, \cite{GriessRyba:survey}, and \cite{CohenWales:survey} for surveys. 
One would like to know also how many conjugacy classes there are for each of these finite subgroups, for which the interesting case of the alternating group on 5 symbols was resolved in \cite{Lusztig:A5}.  These questions can be viewed as a case of a natural generalization of the classification of the finite simple subgroups of $\SO(3)$, which amounts to a classification of the Platonic solids as in \cite[Chap.~I]{Klein:ico}.  Specifically, the Platonic solids correspond to the embeddings of the alternating group on 4 or 5 letters, $\PSL(2,3)$ and $\PSL(2,5)$, and the symmetric group on 5 letters, $\PGL(2,3)$, in $\SO(3)$, and these embeddings are part of a series of embeddings of such subgroups in simple Lie groups including the case of $E_8(k)$, see  \cite{Se:plon} and \cite{Se:sgf}.

\subsection*{Vanishing of trace forms} For a representation $\rho \!: \e_{8,F} \to \gl(V)$ for some $V$, the map $b_\rho \!: (x, y) \mapsto \Tr(\rho(x) \rho(y))$ defines an $\e_{8,F}$-invariant symmetric bilinear form on $\e_{8,F}$.  (When $\rho$ is the adjoint representation, $b_\rho$ is the Killing form.)  In the 1960s, motivated by then-current approaches to studying Lie algebras over fields of prime characteristic, it was an open problem to determine whether $b_\rho$ is identically zero for all $\rho$ when $\car F = 5$, see \cite[p.~48]{Sel:mod}, \cite[p.~554]{Block:trace}, or \cite[p.~544]{BlockZ}.
It is indeed always zero, and for $\car F = 2, 3$ as well, see  \cite{G:vanish}.

\subsection*{Essential dimension}
The essential dimension $\ed(G)$ of an algebraic group $G$ is a non-negative integer that declares, roughly speaking, the number of parameters needed to specify a $G$-torsor, see \cite{Rei:what}, \cite{Rei:ICM}, or \cite{M:ed} for a formal definition and survey of what is known.  In the case of $E_8$, this equals the number of parameters needed to specify a group of type $E_8$.  So far, we only know bounds on $\ed(E_8)$ and the bounds we know are quite weak. 
 Specifically, over $\C$ we have
\[
9 \le \ed(E_8) \le 231,
\]
where the lower bound is from \cite{GiRei} and the upper bound is from \cite{Lemire}.  The distance between the upper and lower bounds is remarkable.  In contrast, for the other simply connected exceptional groups over $\C$, one knows by \cite{MacD:F4} and \cite{LoetscherMacD} that 
\[
\ed(G_2) = 3, \quad 5 \le \ed(F_4) \le 7, \quad 4 \le \ed(E_6) \le 8, \quad \text{and} \quad 7 \le \ed(E_7) \le 11,
\]
which are all much closer.
Determining $\ed(E_8)$ will require new techniques.

\subsection*{Yet more topics} We have furthermore omitted any discussion of:
\begin{itemize}
\item Relations with vertex operator algebras as in \cite{FFR:VOA} and \cite{FLM:VOA}.
\item The adjoint representation of $E_8$ is in some sense unique among irreducible representations of simple algebraic groups, in that it is the only non-minuscule standard module that is irreducible in all characteristics, see \cite{GGN}.
\item The Kneser-Tits Problem as described in \cite{Gille:KT}.  One of the remaining open cases for $E_8$ was recently settled in \cite{PTW}, but another rank 2 case remains.
\item Affine buildings with residues of type $E_8$ as in \cite{MPW:descent}.
\end{itemize}


\bibliographystyle{amsalpha}
\bibliography{skip_master}

\end{document}